\documentclass[12pt]{amsart}

\usepackage{datetime}
\usepackage{times}
\usepackage{amsthm}
\usepackage{amssymb}
\usepackage{amsmath}
\usepackage{amsfonts} 
\usepackage{amsthm}
\usepackage{bm}
\usepackage{color}
\usepackage{graphicx}
\usepackage[all]{xy}

\newcommand{\excise}[1]{}

\newtheorem{theorem}{Theorem}[section]
\newtheorem{lemma}[theorem]{Lemma}

\newtheorem{proposition}[theorem]{Proposition}

\theoremstyle{definition}
\newtheorem{example}[theorem]{Example}
\newtheorem{remark}[theorem]{Remark}

\newtheorem{definition}[theorem]{Definition}

        {\begin{list}
                {\noindent\makebox[0mm][r]{\arabic{enumi}.}}
                {\leftmargin=5.5ex \usecounter{enumi}}
        }
        {\end{list}}

        {\begin{list}
                {\noindent\makebox[0mm][r]{(\roman{enumi})}}
                {\leftmargin=5.5ex \usecounter{enumi}}
        }
        {\end{list}}

\newcommand{\baseRing}[1]{\ensuremath{\mathbb{#1}}}
\newcommand{\Z}{\baseRing{Z}}
\newcommand{\C}{\baseRing{C}}
\newcommand{\N}{\baseRing{N}}
\newcommand{\R}{\baseRing{R}}
\newcommand{\Q}{\baseRing{Q}}

\def\<{\langle}
\def\>{\rangle}
\def\0{\mathbf{0}}

\def\CC{{\mathbb C}}

\def\HH{{\mathcal H}}
\def\KK{{\mathcal K}}

\def\RR{{\mathbb R}}

\def\ZZ{{\mathbb Z}}

\def\cC{{\mathcal C}}
\def\cD{{\mathcal D}}

\newcommand\cH{{\mathcal H}}

\def\cV{{\mathcal V}}
\def\cZ{{\mathcal Z}}

\def\qdeg{{\rm qdeg}}
\def\tdeg{{\rm tdeg}}

\newcommand{\Ch}{{\rm Ch}}
\newcommand{\CCh}{{\rm CCh}}

\def\rank{{\rm rank\ }}

\numberwithin{equation}{section}
\begin{document}


\title{On irregular binomial $D$--modules}\thanks{Partially
supported by MTM2007-64509, MTM2010-19336 and FEDER, FQM333. MCFF
supported by a grant from Iceland, Liechtenstein and Norway through
the EEA Financial Mechanism. Supported and coordinated by
Universidad Complutense de Madrid.}

\author{Mar\'ia-Cruz Fern\'andez-Fern\'andez}
\address{Centre of Mathematics for Applications (University of Oslo)
and Department of Algebra (University of Sevilla).}
\email{mcferfer@algebra.us.es}

\author{Francisco-Jes\'us Castro-Jim\'enez}
\address{Department of Algebra (University of Sevilla).} \email{castro@algebra.us.es}

\maketitle

\begin{center}
\today
\end{center}

\begin{abstract}
We prove that a holonomic binomial $D$--module $M_A (I,\beta )$ is
regular if and only if certain associated primes of $I$ determined
by the parameter vector $\beta\in \CC^d$ are homogeneous. We further 
describe the slopes of $M_A (I,\beta )$ along a coordinate subspace
in terms of the known slopes of some related hypergeometric
$D$--modules that also depend on $\beta$. When the parameter $\beta$
is generic, we also compute the dimension of the generic stalk of
the irregularity of $M_A(I,\beta)$ along a coordinate hyperplane and
provide some remarks about the construction of its Gevrey solutions.
\end{abstract}

\section{Introduction}
Binomial $D$-modules have been introduced by A. Dickenstein, L.F.
Matusevich and E. Miller in \cite{DMM}. These objects generalize
both GKZ hypergeometric $D$-modules \cite{GGZ87, GZK89} and
(binomial) Horn systems, as treated in \cite{DMM} and
\cite{Sadikov2002}.

Here $D$ stands for the complex Weyl algebra of order $n$, where
$n\geq 0$ is an integer. Elements in $D$ are linear partial
differential operators; such an operator $P$ can be written as a
finite sum
$$P=\sum_{\alpha, \gamma} p_{\alpha \gamma } x^\alpha \partial^\gamma $$
where $p_{\alpha \gamma}\in \C$, $\alpha=(\alpha_1, \ldots,
\alpha_n), \gamma =(\gamma_1,\ldots,\gamma_n)\in \N^n$ and $x^\alpha
= x_1^{\alpha_1} \cdots x_n^{\alpha_n}$, $\partial^\gamma =
\partial_1^{\gamma_1}\cdots
\partial_n^{\gamma_n}$. The partial derivative $\frac{\partial}{\partial
x_i}$ is just denoted by $\partial_i$.

Our input is a pair $(A,\beta)$ where $\beta$ is a vector in $\C^d$
and $A=(a_{ij})\in \Z^{d\times n}$ is a matrix whose columns
$a_1,\ldots,a_n$ span the $\Z$-module $\Z^d$. We also assume that
all $a_i\neq 0$ and that the cone generated by the columns in $\R^n$
contains no lines (one says in this case that this cone is {\em
pointed}).

The polynomial ring $\C[\partial]:=\C[\partial_1,\ldots,\partial_n]$
is a subring of the Weyl algebra $D$. The matrix $A$ induces a
$\Z^d$-grading on $\C[\partial]$ (also called the $A$-grading) by
defining $\deg(\partial_i)=-a_i$.

A binomial in $\C[\partial]$ is a polynomial with at most two
monomial terms. An ideal $I$ in $\C[\partial]$ is said to be
binomial is it is generated by binomials. We also say that the ideal
$I$ is an $A$-graded ideal if it is generated by $A$-homogenous
elements (equivalently if for every polynomial in $I$ all its
$A$-graded components are also in $I$).

The matrix $A$ also induces a $\Z^d$-grading on the Weyl algebra $D$
(also called the $A$-grading) by defining $\deg(\partial_i)=-a_i$
and $\deg(x_i)=a_i$.

To the matrix $A$ one associates the toric ideal $I_A \subset
\C[\partial]$ generated by the family of binomials $\partial^u
-\partial^v$ where $u,v\in \N^n$ and $Au=Av$. The ideal $I_A$ is a
prime $A$-graded ideal.

Recall that to the pair $(A,\beta)$ one can associate  the GKZ
hypergeometric ideal $$H_A(\beta)=DI_A + D(E_1-\beta_1,\ldots,
E_d-\beta_d) $$ where $E_i= \sum_{j=1}^n a_{ij}x_j\partial_j$
is the $i^{\rm{th}}$ Euler operator associated with $A$. The
corresponding GKZ hypergeometric $D$--module is nothing but the
quotient (left) $D$--module $M_A(\beta):= \frac{D}{H_A(\beta)}$, \cite{GGZ87}, \cite{GZK89}.

Following \cite{DMM}, for any $A$--graded binomial ideal $I\subset
\C[\partial]$ we denote by $H_A(I,\beta)$ the $A$-graded left ideal
in $D$ defined by $$H_A(I,\beta)=DI + D(E_1-\beta_1,\ldots,
E_d-\beta_d).$$ The binomial $D$--module associated with the
triple $(A,\beta, I)$ is, by definition, the quotient
$M_A(I,\beta):=\frac{D}{H_A(I,\beta)}$. Notice that the ideal
$H_A(I_A,\beta)$ is nothing but the GKZ hypergeometric ideal
$H_A(\beta)$.

In \cite{DMM} the authors have answered  essential questions about
binomial $D$--modules. The main treated questions are related to the
holonomicity of the systems and to the dimension of their
holomorphic solution space around a non singular point. In
particular, in  \cite[Theorem 6.3]{DMM} they prove  that the
holonomicity of $M_A(I,\beta)$ is equivalent to regular holonomicity
when $I$ is standard $\ZZ$-graded (i.e., the row-span of $A$
contains the vector $(1,\ldots,1)$). However, it turns out that the
final sentence in \cite[Theorem 6.3]{DMM}, stating that the regular
holonomicity of $M_A(I,\beta)$ for a given parameter $\beta$ implies
standard homogeneity of the ideal $I$, is true for binomial Horn
systems but it is not for general binomial $D$--modules. This is
shown by Examples \ref{counterexample1} and \ref{counterexample2}.

These two Examples are different in nature. More precisely, the
system $M_A (I, \beta)$ considered in Example \ref{counterexample1}
is regular holonomic for parameters $\beta$ outside a certain line
in the affine complex plane and irregular otherwise, while the
system considered in Example \ref{counterexample2} is regular
holonomic for all parameters despite the fact that the binomial
ideal $I$ is not homogeneous with respect to the standard
$\Z$--grading. This is a surprising phenomenon since it is not
allowed neither for GKZ hypergeometric systems nor for binomial Horn
systems.

We further provide, in Theorem \ref{regularity},  a characterization of the regular holonomicity of
a system $M_A(I,\beta)$ that  improves the
above mentioned result of \cite[Th. 6.3]{DMM}.

A central  question in the study of the irregularity of a
holonomic $D$-module $M$ is the computation of its slopes along
smooth hypersurfaces (see \cite{Meb90} and \cite{Laurent-Mebkhout}).
On the other hand, the Gevrey solutions of $M$ along smooth
hypersurfaces are closely related with the irregularity and the
slopes of $M$. More precisely, the classes of these Gevrey series
solutions of $M$ modulo convergent series define the 0-th cohomology
group of the irregularity of $M$ \cite[D\'{e}finition 6.3.1]{Meb90}.

In Section \ref{section-slopes} we describe the $L$--characteristic
variety and the slopes of $M_A(I,\beta)$ along coordinate subspaces
in terms of the same objects of the binomial $D$--modules associated
with some of the {\em toral} primes of the ideal $I$ determined by
$\beta$  (see Theorem \ref{slopes}). The binomial $D$--module
associated with a toral prime is essentially a GKZ hypergeometric
system and the $L$--characteristic variety and the slopes along
coordinate subspaces of such a system are completely described in
\cite{SW} in a combinatorial way (see also \cite{Castro-Takayama-03}
and \cite{Hartillo03,Hartillo04} for the cases $d=1$ and $n=d+1$).

Gevrey solutions of hypergeometric systems along coordinate
subspaces are described in \cite{F} (see also \cite{FC1},
\cite{FC2}). In Section \ref{Section-Gevrey} we compute the
dimension of the generic stalk of the irregularity of binomial
$D$-modules when the parameter is generic (see Theorem
\ref{theorem-dim-irreg}). We finally give a procedure to compute
Gevrey solutions of $M_A(I,\beta)$ by using known results in the
hypergeometric case (\cite{GZK89}, \cite{SST} and \cite{F}).

We are grateful to Ezra Miller for his useful suggestions and comments.

\section{Preliminaries on Euler--Koszul homology, binomial primary decomposition and toral and Andean modules}

We review here some definitions, notations and results of \cite{ES},
\cite{MMW}, \cite{DMM} and \cite{DMM2} that will be used in the
sequel.

We will denote $R=\C[\partial]$. Recall that the $A$--grading on the
ring $R$ is defined by $\deg(\partial_j)=-a_j$ where $a_j$ is the
$j^{{\rm th}}$-column of $A$. This $A$--grading on $R$ can be
extended to the ring $D$ by setting $\deg(x_j)=a_j$.

\begin{definition}{\rm \cite[Definition 2.4]{DMM}}
Let $V=\oplus_{\alpha\in \Z^d} V_\alpha$ be an $A$-graded $R$-module. The set of true degrees of $V$
is
$$\tdeg(V )=\{\alpha \in \Z^d : \; V_{\alpha} \neq 0\}$$ The set
of quasidegrees of $V$ is the Zariski closure in $\C^d$ of $\tdeg(V
)$.
\end{definition}

\noindent {\em Euler-Koszul complex $\KK_\bullet(E-\beta; V)$
associated with an $A$-graded $R$--module $V$}.

For any $A$--graded left $D$--module $N=\oplus_{\alpha\in \ZZ^d}
N_\alpha$ we denote $\deg_i(y)=\alpha_i$ if $y\in N_\alpha$.

The map $E_i-\beta_i : N_\alpha \rightarrow N_\alpha$ defined by
$(E_i-\beta_i)(y)= (E_i-\beta_i -\alpha_i)y$ can be extended (by
$\C$--linearity) to a morphism of left $D$--modules $E_i-\beta_i : N
\rightarrow N$. We denote by $E-\beta$ the sequence of commuting
endomorphisms $E_1-\beta_1, \ldots, E_d-\beta_d$. This allows us to
consider the Koszul complex $K_\bullet(E-\beta, N)$ which is
concentrated in homological degrees $d$ to $0$.

\begin{definition} {\rm \cite[Definition 4.2]{MMW}} For any
$\beta \in \C^d$ and any $A$-graded $R$--module $V$, the
Euler-Koszul complex $\KK_\bullet(E-\beta, V)$ is the Koszul complex
$K_\bullet(E-\beta, D\otimes_R V)$. The $i^{{\rm th}}$ Euler-Kozsul
homology of $V$, denoted by $\HH_i(E-\beta,V)$, is the homology
$H_i(\KK_\bullet(E-\beta, V))$.
\end{definition}

\begin{remark}\label{Euler-Koszul-shift}
Recall that we have the $A$--graded isomorphism $\HH_i(E-\beta,V)
(\alpha ) \simeq \HH_i(E-\beta +\alpha ,V)(\alpha )$ for all $\alpha
\in \Z^d$ \cite{MMW}. Here $V(\alpha )$ is nothing but $V$ with the
shifted $A$--grading $V(\alpha )_{\gamma }=V_{\alpha +\gamma }$ for
all $\gamma \in \Z^d$.
\end{remark}

\noindent {\em Binomial primary decomposition for binomial ideals.}

We recall from \cite{ES} that for any sublattice $\Lambda\subset
\ZZ^n$ and any partial character $\rho : \Lambda \rightarrow \C^*$,
the corresponding associated binomial ideal is $$I_\rho = \langle
\partial^{u_+} - \rho(u) \partial^{u_-} \, \vert \, u=u_+-u_-\in \Lambda\rangle$$
where $u_+$ and $u_-$ are in $\N^n$ and they have disjoint supports.
The ideal $I_\rho$ is prime if and only if $\Lambda$ is a saturated
sublattice of $\Z^n$  ({\em i.e.} $\Lambda=\Q \Lambda \cap \Z^n$).
We know from \cite[Corollary 2.6]{ES} that any binomial prime ideal
in $R$ has the form $I_{\rho,J}:=I_\rho + {\mathfrak m}_J$ (where
${\mathfrak m}_J = \langle \partial_j\, \vert \, j\not\in J\rangle$)
for some partial character $\rho $ whose domain is a saturated
sublattice of $\ZZ^J$ and some $J\subset \{1,\ldots,n\}$.

For any $J\subset \{1,\ldots,n\}$ we denote by $\partial_J$ the
monomial $\prod_{j\in J} \partial_j$.

\begin{theorem}\label{binomial-primary-decomposition} {\rm \cite[Theorem 3.2]{DMM2}} Fix a binomial ideal $I$
in $R$. Each associated binomial prime $I_{\rho,J}$ has an
explicitly defined monomial ideal $U_{\rho,J}$ such that
$$I=\bigcap_{I_{\rho,J}\in Ass(I)} \cC_{\rho,J}$$ for $\cC_{\rho,J}
= \left((I+I_\rho):\partial_J^\infty\right) + U_{\rho,J}$, is a
primary decomposition of $I$ as an intersection of $A$--graded
primary binomial ideals.
\end{theorem}

\noindent {\em Toral and Andean modules}.

In \cite[Definition 4.3]{DMM2} a finitely generated $A$-graded
$R$--module $V=\oplus V_\alpha$ is said to be {\em toral} if its
Hilbert function $H_V$ (defined by $H_V(\alpha)=\dim_\C V_\alpha$
for $\alpha \in \Z^d$) is bounded above.

With the notations above, a $R$--module of type $R/I_{\rho, J}$ is
toral if and only if its Krull dimension equals the rank of the
matrix $A_J$ (see \cite[Lemma 3.4]{DMM}). Here $A_J$ is the
submatrix of $A$ whose columns are indexed by $J$. In this case the
module $R/\cC_{\rho, J}$ is toral and we say that the ideal
$I_{\rho,J}$ is a toral prime and  $\cC_{\rho,J}$ is a toral primary
component.

If $\dim (R/I_{\rho,J}) \not= \rank(A_J)$ then the module
$R/\cC_{\rho, J}$ is said to be {\em Andean}, the ideal $I_{\rho,J}$
is an {\em Andean} prime and $\cC_{\rho,J}$ is an {\em Andean}
primary component.

An $A$--graded $R$--module $V$ is said to be {\em natively toral} if
there exist a binomial toral prime ideal $I_{\rho,J}$ and an element
$\alpha\in \Z^d$ such that $V(\alpha)$ is isomorphic to $R/I_{\rho,
J}$ as $A$--graded modules (see \cite[Definition 4.1]{DMM}).

\begin{proposition}{{\rm \cite[Proposition 4.2]{DMM}}}
An $A$--graded $R$--module $V$ is toral if and only if it has a
filtration $$0=V_0 \subset V_1 \subset \cdots \subset V_{\ell -1}
\subset V_\ell =V$$ whose successive quotients $V_k/V_{k-1}$ are all
natively toral.
\end{proposition}

Such a filtration on $V$ is called a {\em toral filtration}.

Following \cite[Definition 5.1]{DMM} an $A$-graded $R$-module $V$ is said
to be {\em natively Andean}  if there is an $\alpha \in \Z^d$ and an
Andean quotient ring $R/I_{\rho,J}$ over which $V (\alpha)$ is
torsion-free of rank 1 and admits a $\Z^J/\Lambda$-grading that refines
the $A$-grading via $\Z^J/\Lambda \rightarrow \Z^d = \Z A$, where $\rho$
is defined on $\Lambda \subset \Z^J$. Moreover, if $V$ has a finite
filtration  $$0=V_0 \subset V_1 \subset \cdots \subset V_{\ell -1}
\subset V_\ell =V$$ whose successive quotients $V_k/V_{k-1}$ are all
natively Andean, then $V$ is Andean (see \cite[Section 5]{DMM}).

In \cite[Example 4.6]{DMM2} it is proven that the quotient
$R/C_{\rho,J}$ is Andean for any Andean primary component
$C_{\rho,J}$  of any $A$-graded binomial ideal.

We finish this section with the definition and a result about the
so-called {\em Andean arrangement} associated with an $A$-graded
binomial ideal $I$ in $R$. Let us fix an irredundant primary
decomposition $$I=\bigcap_{I_{\rho,J}\in Ass(I)} \cC_{\rho,J}$$ as
in Theorem \ref{binomial-primary-decomposition}.

\begin{definition}\cite[Definition 6.1]{DMM} The Andean arrangement $\cZ_{\rm Andean}(I)$ is the union
of the quasidegree sets $\qdeg(R/C_{\rho,J} )$ for the Andean
primary components $C_{\rho,J}$ of $I$.
\end{definition}

From \cite[Lemma 6.2]{DMM} the Andean arrangement $\cZ_{\rm Andean}(I)$ is a union of
finitely many integer translates of the subspaces $\C A_J \subset
\C^n$ for which there is an Andean associated prime $I_{\rho,J}$.

From \cite[Theorem 6.3]{DMM} we have that the binomial $D$--module $M_A(I,\beta)$ is holonomic if and only if
$-\beta \notin \cZ_{\rm Andean}(I)$.

\section{Characterizing regular holonomic binomial $D$--modules}

Let $I$ be an $A$--graded binomial ideal and fix a binomial primary
decomposition $I=\cap_{\rho ,J} C_{\rho , J}$ where $ C_{\rho , J}$
is a $I_{\rho ,J}$--primary binomial ideal.

Let us consider the ideal $$I_{\beta}:=\bigcap_{-\beta \in
\qdeg(R/C_{\rho ,J})} C_{\rho ,J}$$ i.e., the intersection of all
the primary components $C_{\rho ,J}$ of $I$ such that $-\beta$ lies
in the quasidegrees set of the module $R/C_{\rho ,J}$.

\begin{remark}
Notice that if $-\beta  \notin \cZ_{\rm Andean}(I)$ then $R
/I_{\beta}$ is contained in the toral direct sum
$$\bigoplus_{-\beta \in \qdeg(R/C_{\rho,J})} R /C_{\rho,J}$$
and so it is a toral module.
\end{remark}

The following result generalizes \cite[Proposition 6.4]{DMM}.

\begin{proposition}\label{isomorphism-I-beta}
If $-\beta \notin \mathcal{Z}_{\rm Andean}(I)$ then the natural
surjection $R/I \twoheadrightarrow R/I_{\beta}$ induces a isomorphism in Euler--Koszul
homology
$$\mathcal{H}_i (E-\beta , R/ I) \simeq  \mathcal{H}_i (E-\beta , R/
I_{\beta})$$ for all $i$. In particular, $M_A (I ,\beta )\simeq M_A
(I_{\beta},\beta )$.
\end{proposition}
\begin{proof}
By \cite[Proposition 6.4]{DMM} we have that
$$\mathcal{H}_i (E-\beta , R/ I) \simeq  \mathcal{H}_i (E-\beta , R/
I_{{\rm toral}})$$ for all $i$, where $I_{{\rm toral}}$ denotes the intersection
of all the toral primary components of $I$. Thus, we can assume
without loss of generality that all the primary components of $I$
are toral. The rest of the proof is now analogous to the proof of
\cite[Proposition 6.4]{DMM} if we substitute the ideals $I_{{\rm toral}}$
and $I_{\rm Andean}$ there by the ideals $I_{\beta}$ and
$\overline{I_{\beta}}$ respectively, where
$$\overline{I_{\beta}}=\bigcap_{-\beta \notin \qdeg(R/C_{\rho ,J})} C_{\rho ,J},$$  and the Andean direct sum
$\bigoplus_{I_{\rho , J} {\rm Andean}} R/C_{\rho ,J}$ there by the toral
direct sum $$\bigoplus_{-\beta \notin \qdeg(R/C_{\rho ,
J})} R/C_{\rho ,J}$$ Finally, we can use Lemma 4.3 and Theorem 4.5
in \cite{DMM} in a similar way as \cite[Lemma 5.4]{DMM} is used in
the proof of Proposition 6.4 of \cite{DMM}.
\end{proof}

The following Lemma gives a description of the quasidegrees set of a
toral module of type $R/C_{\rho,J}$. E. Miller has pointed out
that this result follows from Proposition 2.13 and Theorem 2.15 in
\cite{DMM2}. We will include here a slightly different  proof of
this Lemma.

\begin{lemma}\label{qdeg-toral-lemma}
For any $I_{\rho ,J}$--primary toral ideal $C_{\rho ,J}$ the
quasidegrees set of $M=R/C_{\rho ,J}$ equals the union
of at most $\mu_{\rho ,J}$ $\Z^d$--graded translates of $\C A_J$,
where $\mu_{\rho ,J}$ is the multiplicity of $I_{\rho ,J}$ in
$C_{\rho ,J}$. More precisely, for any toral filtration $0=M_0
\subseteq M_1 \subseteq \cdots \subseteq M$ we have that the
quasidegrees set of $M$ is the union of the quasidegrees set of all
the successive quotients $M_i /M_{i-1}$ that are isomorphic to
$\Z^d$--graded translates of $R/I_{\rho ,J}$.
\end{lemma}

\begin{proof}
Since $M$ is toral we have by \cite[Lemma 4.7]{DMM} that $\dim
(\qdeg(M))=\dim M =\rank A_J$. Since $C_{\rho, J}$ is primary, any
zero-divisor of $M$ is nilpotent. For all $j\in J$ we have that
$\partial_j^m \notin C_{\rho ,J}\subseteq I_{\rho}+\mathfrak{m}_J$
and so $\partial_j$ is not a zero-divisor in $M$ for all $j\in J$.
Thus, the true degrees set of $M$ verifies $\tdeg(M)=\tdeg (M)-\N
A_J$. This and the fact that $\dim (\qdeg (M))=\rank A_J$ imply that
there exists $\alpha_1 ,\ldots ,\alpha_r \in \Z^d$ such that $\tdeg
(M)= \cup_{i=1}^r (\alpha_i -\N A_J)$ and
\begin{equation}
\qdeg (M)= \bigcup_{i=1}^r (\alpha_i + \C  A_J) \label{qdeg-toral}
\end{equation}
Consider now a toral filtration $0=M_0 \subseteq M_1 \subseteq
\cdots \subseteq M$. We know that there are exactly $\mu_{\rho ,J}$
different values of $i$ such that $M_i /M_{i-1}\simeq R/I_{\rho ,J}(\gamma_i )$ for some $\gamma_i\in \ZZ^d$.
For the other successive quotients $M_l
/M_{l-1}\simeq R/I_{\rho_l ,J_l}(\gamma_l )$ we have that $
I_{\rho_l ,J_l}$ is a toral prime which properly contains $I_{\rho
,J}$. In particular, we have that $\rank A_{J_l} =\dim R/I_{\rho_l ,J_l} <\dim R/I_{\rho ,J}= \rank A_J$. Since
$\qdeg(R/I_{\rho_l ,J_l})=\C A_{J_l}$ has dimension $\rank A_{J_l}
<\rank A_J$ and $\qdeg(M)=\bigcup_{i} \qdeg(M_i /M_{i-1})$ we have by
(\ref{qdeg-toral}) that the quasidegrees set of any $M_i/M_{i-1}$ is
contained in the quasidegrees set of some $M_j /M_{j-1}\simeq
R/I_{\rho ,J}(\gamma_j )$. In particular $r\leq \mu_{\rho ,J}$ and
each affine subspace $(\alpha_i +\C A_J)$ in (\ref{qdeg-toral}) is
the quasidegrees set of some $M_j/M_{j-1}\simeq R/I_{\rho
,J}(\gamma_j )$.
\end{proof}

\begin{remark}\label{remark-toral-hypergeometric}
Notice that $H_A (I_{\rho ,J} ,\beta)= D H_{A_J} (I_{\rho},\beta) +
D (\partial_j :\; j\notin J)$. In addition, if $I_{\rho ,J}$ is
toral then the $D_J$--module $M_{A_J}(I_{\rho},\beta)$ is isomorphic
to the hypergeometric system $M_{A_J}(\beta)$ via an $A$--graded
isomorphism of $D_J$--modules induced by rescaling the variables
$x_j$, $j\in J$, using the character $\rho$. Thus we can apply most
of the well-knows results for hypergeometric systems to $M_A
(I_{\rho ,J} ,\beta)$ (with $I_{\rho ,J}$ a toral prime) in an
appropriated form.
\end{remark}

\begin{lemma}\label{regularity-natively-toral}
If $I_{\rho ,J}$ is toral and $-\beta \in \qdeg(R/I_{\rho ,J})$ the following conditions are equivalent:
\begin{enumerate}
\item[i)] $\mathcal{H}_i (E-\beta, R/I_{\rho ,J})$ is regular holonomic for all $i$.
\item[ii)] $\mathcal{H}_0 (E-\beta,R/I_{\rho ,J})$ is regular holonomic.
\item[iii)] $I_{\rho ,J}$ is homogeneous (equivalently $A_J$ is
homogeneous).
\end{enumerate}
\end{lemma}
\begin{proof}
$i)\Rightarrow ii)$ is obvious, $ii)\Rightarrow iii)$ follows
straightforward from \cite[Corollary 3.16]{SW} and $iii)\Rightarrow
i)$ is a particular case of the last statement in \cite[Theorem
4.5]{DMM} and it also follows from \cite[Ch. II, 6.2, Thm.]{Hotta}.
\end{proof}

\begin{remark}\label{remark-regularity-natively-toral}
Recall from \cite[Theorem 4.5]{DMM} that for any toral module $V$ we
have that $-\beta \notin \qdeg V$ if and only if $\mathcal{H}_0
(E-\beta,V)=0$ if and only if $\mathcal{H}_i (E-\beta,V)=0$ for all
$i$. In particular, since the $D$--module $0$ is regular holonomic
it follows that conditions i) and ii) in Lemma
\ref{regularity-natively-toral} are also equivalent without the
condition $-\beta \in \qdeg(R/I_{\rho ,J})$.
\end{remark}

\begin{theorem}\label{regularity}
Let $I\subseteq R$ be an $A$-graded binomial ideal such that $M_A
(I,\beta)$ is holonomic (equivalently, $-\beta \notin
\mathcal{Z}_{\rm Andean}(I)$). The following conditions are
equivalent:
\begin{enumerate}
\item[i)] $\mathcal{H}_i (E-\beta, R/I)$ is regular
holonomic for all $i$.

\item[ii)] $M_A (I,\beta )$ is regular holonomic.

\item[iii)] All the associated toral primes $I_{\rho ,J}$ of $I$
such that $-\beta \in \qdeg(R /C_{\rho ,J})$ are
homogeneous.
\end{enumerate}
\end{theorem}
\begin{proof}
The implication $i)\Rightarrow ii)$ is obvious. Let us prove
$ii)\Rightarrow iii)$. For any toral primary component $C_{\rho , J
}$ of $I$ we have $I\subseteq C_{\rho ,J}$ and so there is a natural
epimorphism $M_A (I, \beta ) \twoheadrightarrow M_A ( C_{\rho ,J}
,\beta )$. Since  $M_A (I,\beta )$ is regular holonomic then $M_A
(C_{\rho ,J},\beta )$ is also regular holonomic. Take a toral
filtration of $M=R/C_{\rho ,J}$, $0\subseteq M_1
\subseteq \cdots \subseteq M_r =M$.  We claim that
\begin{equation}
\mathcal{H}_j (E-\beta, M_i /M_{i-1}) \mbox{ and } \mathcal{H}_0
(E-\beta,M_{i-1}) \mbox{ are regular holonomic} \label{assertion}
\end{equation} for all $i,j$.

Let us prove (\ref{assertion}) by decreasing induction on $i$. For
$i=r$, we have a surjection from the regular holonomic $D$--module
$\mathcal{H}_0 (E-\beta,M_r)=M_A (C_{\rho ,J},\beta )$ to
$\mathcal{H}_0 (E-\beta,M_r /M_{r-1})$ and so it is regular
holonomic too. By Remark \ref{Euler-Koszul-shift}, Lemma
\ref{regularity-natively-toral} and Remark
\ref{remark-regularity-natively-toral} we have that the $D$-module
$\mathcal{H}_j (E-\beta,M_r /M_{r-1})$ is regular holonomic for all
$j$. Since $$\mathcal{H}_1 (E-\beta, M_{r}/M_{r-1}) \longrightarrow
\mathcal{H}_0 (E-\beta, M_{r-1}) \longrightarrow \mathcal{H}_0
(E-\beta, M_r)$$ is exact we have that $\mathcal{H}_0 (E-\beta,
M_{r-1})$ is regular holonomic.

Assume that (\ref{assertion}) holds for some $i=k+1\leq r$ and for
all $j$. We consider the exact sequence $$0\longrightarrow M_{k-1}
\longrightarrow M_k \longrightarrow M_k / M_{k-1} \longrightarrow
0$$ and the following part of the long exact sequence of
Euler-Koszul homology

{\small \begin{equation} \cdots \mathcal{H}_1
(E-\beta, M_{k}/M_{k-1}) \rightarrow \mathcal{H}_0 (E-\beta,
M_{k-1}) \rightarrow \mathcal{H}_0 (E-\beta,  M_k)
\twoheadrightarrow \mathcal{H}_0 (E-\beta, M_k / M_{k-1}).
\label{Euler-Koszul-long}
\end{equation}}

By induction hypothesis $\mathcal{H}_0 (E-\beta,  M_k)$ is regular
holonomic. This implies that $\mathcal{H}_0 (E-\beta, M_k /
M_{k-1})$ is regular holonomic by (\ref{Euler-Koszul-long}).
Applying Remark \ref{Euler-Koszul-shift}, Lemma
\ref{regularity-natively-toral} and Remark
\ref{remark-regularity-natively-toral} we have that $\mathcal{H}_j
(E-\beta, M_k /M_{k-1})$ is regular holonomic for all $j$. Thus, by
(\ref{Euler-Koszul-long}) we have that $\mathcal{H}_0 (E-\beta,
M_{k-1})$ is regular holonomic too and we have finished the induction
proof of (\ref{assertion}).

Assume that $-\beta \in \qdeg(R /C_{\rho ,J})$. By Lemma
\ref{qdeg-toral-lemma} there exists $i$ such that $-\beta$ lies in
the quasidegrees set of $M_i /M_{i-1}\simeq R/I_{\rho
,J}(\gamma_i )$ and we also have by (\ref{assertion}) that
$$\mathcal{H}_0 (E-\beta, M_i /M_{i-1})\simeq \mathcal{H}_0 (E-\beta
+ \gamma_i, R/I_{\rho ,J})(\gamma_i )$$ is a nonzero
regular holonomic $D$-module. Thus, by Lemma
\ref{regularity-natively-toral} we have that $I_{\rho ,J}$ is
homogeneous.

Let us prove $iii)\Rightarrow i)$. By Proposition
\ref{isomorphism-I-beta} we just need to prove that $M_A (I_{\beta},
\beta )$ is regular holonomic. We have that all the associated
primes of $I_{\beta}$ are toral and homogeneous. In particular $M=R/I_{\beta}$ is a toral module and for any toral filtration
of $M$ the successive quotients $M_i /M_{i-1}$ are isomorphic to
some $\Z^d$--graded translate of a quotient $R/I_{\rho_i
,J_i}$ where $I_{\rho_i ,J_i}$ is toral and contains a minimal prime
$I_{\rho ,J}$ of $I_{\beta}$. Such minimal prime is homogeneous by
assumption and so $A_J$ is homogeneous. Since $J_i \subseteq J$ we
have that $A_{J_i}$ and $I_{\rho_i , J_i }$ are homogeneous too.
Now, we just point out that that the proof of the last statement in
\cite[Theorem 4.5]{DMM} still holds for $V=M$ if we don't require
$A$ to be homogenous but all the primes occurring in a  toral
filtration of $M$ to be homogeneous.

\end{proof}

\begin{remark}
Theorem \ref{regularity} shows in particular that the property of a
binomial $D$-module $M_A (I, \beta )$ of being regular (holonomic)
can fail to be constant when $-\beta$ runs outside the Andean
arrangement. This phenomenon is forbidden to binomial Horn systems
$M_A (I(B),\beta)$ (see \cite[Definition 1.5]{DMM}) since the
inclusion $I(B)\subseteq I_A$ induces a surjective morphism
$$\mathcal{H}_0 (E-\beta, I(B))\twoheadrightarrow M_A (\beta )$$ and
then regular holonomicity of $\mathcal{H}_0 (E-\beta,
R/I(B))$ implies regular holonomicity of $M_A (\beta )$,
which is equivalent to the standard homogeneity of $I_A$ by \cite{Hotta, SST, SW}.
\end{remark}

\begin{definition} \label{non-regular-arrangement}
The non-regular arrangement of $I$ (denoted by ${\mathcal Z}_{{\rm non-regular}}(I)$) is the union of the Andean
arrangement of $I$ and the union of quasidegrees sets of the
quotients of $R$ by primary components $C_{\rho ,J}$ of $I$ such
that $I_{\rho ,J}$ is not homogeneous with respect to the standard
grading. So, we have $${\mathcal Z}_{{\rm non-regular}}(I) =
{\mathcal Z}_{{\rm Andean}}(I) \cup \left(\bigcup_{I_{\rho,J} \,
{\rm non \, homogeneous}}  \qdeg(R/C_{\rho,J})\right). $$
\end{definition}

\begin{example}\label{counterexample1}
Consider the ideal $I=\langle \partial_1^2 \partial_2 -\partial_2^2
, \partial_2 \partial_3 , \partial_2 \partial_4 ,\partial_1^2
\partial_3 - \partial_3^2 \partial_4 , \partial_1^2 \partial_4 -\partial_3 \partial_4^2
\rangle$. It is $A$-graded for the matrix
$$A=\left(\begin{array}{cccc}
1 & 2 & 2 & 0 \\
1 & 2 & 0 & 2
\end{array}\right)$$ but $I$ is not standard $\Z$-graded. We have the prime decomposition $I=I_1 \cap I_2 \cap I_3$ where $I_1
= \langle \partial_2 ,\partial_3 ,\partial_4 \rangle$, $I_2 =\langle
\partial_1^2 -\partial_2 ,
\partial_3 ,\partial_4 \rangle$ and $I_3 = \langle
\partial_2 , \partial_1^2 -\partial_3 \partial_4\rangle$ are toral primes of $I$.
In particular $\mathcal{Z}_{\rm Andean}(I)=\emptyset$) and by the
proof of \cite[Proposition 6.6]{DMM} we have that $\mathcal{Z}_{\rm
primary}(I) = \{0\}$ (see \cite[Definition 6.5]{DMM} for the
definition of the primary arrangement $\mathcal{Z}_{\rm
primary}(I)$).

Using \cite[Theorem 6.8]{DMM} we have that $M_A (I,\beta )$ is
isomorphic to the direct sum of $M_A (I_j ,\beta )$ for $j=1,2,3$ if
$\beta \neq 0$. Moreover, $\qdeg(R /I_j )=\C \binom{1}{1}$ for
$j=1,2$ and $\qdeg(R /I_3 )=\C^2$. Thus, for generic parameters
(more precisely for $\beta \in \C^2 \setminus \C \binom{1}{1}$) we
have that $M_A (I,\beta )$ is isomorphic to $M_A (I_3 ,\beta)$ that
is a regular holonomic by  Lemma \ref{regularity-natively-toral}.

On the other hand, there is a surjective morphism from $M_A (I
,\beta)$ to $M_A (I_2 ,\beta)$ and if $\beta \in \C \binom{1}{1}$ we
have that $M_A (I_2 ,\beta)$ is an irregular $D$-module because
$s=2$ is a slope along $x_2=0$. Thus we conclude that $M_A (I ,\beta
)$ is regular holonomic if $\beta \in \C^2\setminus \C \binom{1}{1}$
and it is an irregular holonomic $D$-module when $\beta \in \C
\binom{1}{1}$. In particular, ${\mathcal Z}_{{\rm non-regular}}(I)=
\C \binom{1}{1} \subset \CC^2$. It can also be checked that the
singular locus of $M_A (I, \beta )$ is $\{x_1 x_2 x_3 x_4 (x_1^2-4
x_3 x_4)=0\}$ when $\beta \in \C \binom{1}{1}$ and $\{x_3 x_4
(x_1^2-4 x_3 x_4)=0\}$ otherwise.
\end{example}

\begin{example}\label{counterexample2}
The primary binomial ideal $I=\langle \partial_1 -\partial_2
,\partial_3^4 ,\partial_4^3 ,\partial_3^3-\partial_4^2\rangle$ is
$A$--graded with respect to the matrix $A=(1 \; 1 \; 2 \; 3)$. Note
that $I$ is not homogeneous with respect to the standard
$\Z$-grading. However, its radical ideal $\sqrt{I}=\langle
\partial_1 -\partial_2 ,\partial_3 ,\partial_4 \rangle$ is
homogeneous. Thus, by Theorem \ref{regularity} we have that $M_A
(I,\beta)$ is regular holonomic.
\end{example}

\section{$L$--Characteristic variety and slopes of binomial
$D$--modules}\label{section-slopes}

Let $L$ be the filtration on $D$ defined by a weight vector
$(u,v)\in \R^{2n}$ with $u_i + v_i = c >0$ for some constant $c>0$.

This includes in particular the intermediate filtrations $pF+qV$
between the filtration $F$ by the order of the linear differential
operators and the Kashiwara-Malgrange filtration $V$ along a
coordinate subspace. The filtrations $pF+qV$ are the ones considered
when studying the algebraic slopes of a coherent $D$--module along a
coordinate subspace \cite{Laurent-Mebkhout}.

We will consider the $L$--characteristic variety $\Ch^L(N)$ of a
finitely generated $D$--module $N$ on $\C^n$ defined as the support
of $\operatorname{gr}^L N$ in $T^{\ast} \C^n$ (see e.g.
\cite{Laurent-ens-87}, \cite[Definition 3.1]{SW}). We recall that in
fact for $L=p F + qV$ this is a global algebraic version of
Laurent's microcharacteristic variety of type $s=p/q + 1$ \cite[\S
3.2]{Laurent-ens-87} (see also \cite[Remark 3.3]{SW}).

The $L$-characteristic variety and the slopes of a hypergeometric
$D$-module $M_A(\beta)$ are controlled by the so-called
$(A,L)$--umbrella \cite{SW}. Let us recall its definition in the
special case when $v_i > 0$ for all $i$. We denote by $\Delta_A^L$
the convex hull of $\{0,a_1^L,\ldots,a_n^L\}$ where $a_j^L =
\frac{1}{v_j}a_j$. The $(A,L)$-umbrella is the set $\Phi_A^L$ of
faces of $\Delta_A^L$ which do not contain 0. The empty face is in
$\Phi_A^L$. One identifies $\tau\in\Phi_A^L$ with $\{j | a_j^L \in
\tau \}$, or with $\{a_j | a_j^L\in \tau\}$, or with the
corresponding submatrix $A_\tau$  of $A$.

By \cite[Corollary 4.17]{SW} the $L$-characteristic variety of a
hypergeometric $D$--module $M_A(\beta)$ is
\begin{equation} \label{L-characteristic-variety} \Ch^L(M_A(\beta)) = \bigcup_{\tau\in \Phi_A^L}
\overline{C_A^\tau}\end{equation}  where $\overline{C_A^\tau}$ is
the Zariski closure in $T^*\CC^n$ of the conormal space to the orbit
$O_A^\tau\subset T^{\ast}_0 \CC^n = \CC^n$ corresponding to the face
$\tau$. In particular $\Ch^L(M_A(\beta))$ is independent of $\beta$.
By definition we have the equality $O_A^\tau:=(\CC^*)^d\cdot {\bf
1}_A^\tau$ where ${\bf 1}_A^\tau\in \N^n$ is defined by $({\bf
1}_A^\tau)_j=1$ if $j\in \tau$ and $({\bf 1}_A^\tau)_j=0$ otherwise.
The action of the torus is given with respect to the matrix $A$. If
the filtration given by $L$ equals the $F$-filtration (i.e. the
order filtration) then this description of the $F$--characteristic
variety coincides with a result of \cite[Lemmas 3.1 and 3.2]{Ado94}.

\begin{proposition}\label{slopes-primary}
If $M$ is a $I_{\rho ,J}$--coprimary toral module and $-\beta \in
\qdeg(M)$ then the $L$--characteristic variety of $\mathcal{H}_0
(E-\beta, M)$ is the $L$--characteristic variety of $M_A ( I_{\rho ,
J}, 0)$. In particular, the set of slopes of $\mathcal{H}_0
(E-\beta, M)$ along a coordinate subspace in $\C^n$ coincide with
the ones of $M_A ( I_{\rho , J}, 0)$.
\end{proposition}

\begin{proof}
Since $M$ is $I_{\rho ,J}$--coprimary there exists $m\geq 0$ such
that $I_{\rho ,J}^m$ annihilates $M$. Consider a set of
$A$--homogeneous elements $m_1 ,\ldots , m_k \in M$ generating $M$
as $R$--module. This leads to a natural $A$--graded surjection $
\bigoplus_{i=1}^k R/I_{\rho ,J}^m (- \deg ( m_i ) )
\twoheadrightarrow M$. In particular, there is a surjective morphism
of $D$-modules
$$\bigoplus_{i=1}^k \mathcal{H}_0 (E-\beta, R/I_{\rho ,J}^m
(- \deg (m_i ) )) \twoheadrightarrow \mathcal{H}_0 (E-\beta, M)$$ inducing the inclusion:
$$\operatorname{Ch}^L (\mathcal{H}_0
(E-\beta, M))\subseteq  \mathcal{V}( \operatorname{in}_L (I_{\rho
,J}^m), A x \xi)=\mathcal{V}( \operatorname{in}_L (I_{\rho}), A_J
x_J \xi_J , \xi_j : j\notin J ).$$ Here $(x,\xi)$ stands for the
coordinates in the cotangent space $T^*\CC^n$, $x\xi
=(x_1\xi_1,\ldots,x_n\xi_n)$ and $\cV$ is the zero set in $T^*\CC^n$
of the corresponding ideal.

The equality $\operatorname{Ch}^L
( M_A (I_{\rho ,J}, 0 ))=\mathcal{V}( \operatorname{in}_L
(I_{\rho}), A_J x_J \xi_J , \xi_j : j\notin J )$ follows from
\cite[(3.2.2) and Corollary 4.17]{SW}. Thus,
\begin{equation}
\operatorname{Ch}^L (\mathcal{H}_0 (E-\beta, M))\subseteq
\operatorname{Ch}^L ( M_A (I_{\rho ,J}, 0 ))
\label{inclusion-charac-var-primary}
\end{equation}

Let us now prove  the equality
\begin{equation}
\operatorname{Ch}^L ( \mathcal{H}_0 (E-\beta, M))=
\operatorname{Ch}^L ( M_A (I_{\rho ,J}, 0
))\label{charac-var-primary}
\end{equation} by induction on the length $r$ of a toral filtration $0=M_0
\subsetneq M_1 \subsetneq \cdots \subsetneq M_r =M$ of $M$.

If $r=1$ we have that $M\simeq R/I_{\rho , J} (\gamma
)$ for some $\gamma \in \Z^d$ and $-\beta \in \qdeg(M)$ means that
$-\beta + \gamma \in \qdeg(R/I_{\rho ,J})=\C A_J$. Thus,
$\mathcal{H}_0 (E-\beta, M )\simeq M_A (I_{\rho ,J}, \beta -\gamma
)$ and we have (\ref{charac-var-primary}) because the
$L$--characteristic variety of $M_A (I_{\rho ,J}, \beta ' )$ is
independent of $\beta ' \in - \qdeg(R/I_{\rho ,J})$ by
the results in \cite{SW}.

Assume by induction that we have (\ref{charac-var-primary}) for all
toral $I_{\rho ,J}$--coprimary modules $M$ with a toral filtration
of length $r$ such that $-\beta \in \qdeg(M)$.

Let $M$ be a $I_{\rho ,J}$--coprimary toral module with toral
filtration of length $r+1$, i.e. $0=M_0 \subsetneq M_1 \subseteq
\cdots \subsetneq M_{r+1} =M$. From the exact sequence
$$0\longrightarrow M_{r} \longrightarrow M \longrightarrow M / M_{r}
\longrightarrow 0$$ we obtain the long exact sequence of
Euler--Koszul homology {\small $$\cdots \longrightarrow \mathcal{H}_1
(E-\beta, M/M_r) \longrightarrow \mathcal{H}_0 (E-\beta, M_r)
\longrightarrow \mathcal{H}_0 ( E-\beta, M ) \longrightarrow
\mathcal{H}_0 (E-\beta, M / M_r) \longrightarrow 0.$$} Now, we need to
distinguish two cases.

Assume first that $-\beta \notin \qdeg(M /M_{r})$. Thus,
$\mathcal{H}_j (E-\beta, M / M_r)=0$ for all $j$ by \cite[Theorem
4.5]{DMM} and we have that $\mathcal{H}_0 (E-\beta, M_r)\simeq
\mathcal{H}_0 (E-\beta, M )$ so they both have the same
$L$--characteristic variety. Notice that the fact that $-\beta \in
\qdeg M \setminus \qdeg(M /M_{r})$ along with Lemma
\ref{qdeg-toral-lemma}  guarantees that there exists some $i\leq r$
such that $M_i /M_{i-1}\simeq R/I_{\rho ,J} (\gamma_i )$. This
implies that $M_r$ is also $I_{\rho ,J}$--coprimary and we can apply
the induction hypothesis.

Assume now that $-\beta \in \qdeg(M /M_{r})$. In this case we still
have that the $L$--characteristic variety of $\mathcal{H}_0
(E-\beta, M / M_{r})$ is contained in the $L$--characteristic
variety of $\mathcal{H}_0 ( E-\beta, M )$. If $M /M_{r} \simeq R/I_{\rho ,J}(\gamma )$ we have that $\operatorname{Ch}^L
(M_A (I_{\rho ,J}, 0))\subseteq \operatorname{Ch}^L (\mathcal{H}_0 (
E-\beta, M ))$ and using (\ref{inclusion-charac-var-primary}) we get
(\ref{charac-var-primary}).

We are left with the case when $-\beta \in \qdeg(M /M_{r})$ and
$M/M_r \simeq R/ I_{\rho ' , J '} (\gamma )$ with $I_{\rho ,
J}\subsetneq I_{\rho ' ,J '}$. This implies that $M_{r}$ is also
$I_{\rho ,J}$--coprimary. Moreover, it is clear that $-\beta \in
\qdeg(M_r )$ by using Lemma \ref{qdeg-toral-lemma}. Thus, we have by
induction hypothesis that the $L$--characteristic variety of
$\mathcal{H}_0 (E-\beta, M_{r})$ is the $L$--characteristic variety
of $M_A (I_{\rho ,J},0)$.

Assume to the contrary that there exists an irreducible component
$C$ of the $L$--characteristic variety of $M_A (I_{\rho ,J},0)$ that
is not contained in the $L$--characteristic variety of
$\mathcal{H}_0 (E-\beta, M)$. This implies that $C$ is not contained
in  $\operatorname{Ch}^L (\mathcal{H}_0 (E-\beta, M/M_{r}))$, i.e.
the multiplicity $\mu_{A,0}^{L,C}(M/M_{r},\beta)$  of $C$ in the
$L$-characteristic cycle of $\mathcal{H}_0 (E-\beta, M/M_{r})$ is
zero (see \cite[Definition 4.7]{SW}). As a consequence, the
multiplicity $\mu_{A,i}^{L,C}(M/M_{r},\beta)$ of $C$ in the
$L$-characteristic cycle of $\mathcal{H}_i (E-\beta, M/M_{r})$ is
zero for all $i\geq 0$ because we can use an adapted version of
\cite[Theorems 4.11 and 4.16]{SW} since $M/M_r $ is a module of the
form $R/(I_{A_{J'}}+\mathfrak{m}_{J'})(\gamma )$ after rescaling the
variables via $\rho$. Now, using the long exact sequence of
Euler--Koszul homology and the additivity of the $L$--characteristic
cycle we conclude that
$\mu_{A,i}^{L,C}(M,\beta)=\mu_{A,i}^{L,C}(M_r,\beta )$ for all
$i\geq 0$. In particular we have that $\mu_{A,0}^{L,C}(M,\beta)>0$
and thus $C$ is contained in the $L$--characteristic variety of
$\mathcal{H}_0 (E-\beta, M)$. We conclude that the
$L$--characteristic variety of $M_A (I_{\rho ,J},0)$ is contained in
the $L$--characteristic variety of $\mathcal{H}_0 (E-\beta, M)$ and
this finishes the induction proof.
\end{proof}

The following result is well known. We include a proof for the sake
of completeness.

\begin{lemma}\label{radical-initial-intersection}
Let $I_1 ,\dots ,I_r$ be a sequence of ideals in $R$ and
$\omega \in \R^n$ a weight vector. Then
\begin{equation}
\cap_{j=1}^r \sqrt{\operatorname{in}_{\omega}(I_j)} =
\sqrt{\operatorname{in}_{\omega} (\cap_{j} I_j )}
\end{equation}
\end{lemma}
\begin{proof}
The inclusion $\operatorname{in}_{\omega}(\cap_{j} I_j )\subseteq
\cap_{j=1}^r \operatorname{in}_{\omega}(I_j)$ is obvious and then we
can take radicals.

Let us see that $\cap_{j=1}^r
\operatorname{in}_{\omega}(I_j)\subseteq
\sqrt{\operatorname{in}_{\omega}(\cap_{j} I_j )}$.  Let us consider
an $\omega$--homogeneous element $f$ in $\cap_{j=1}^r
\operatorname{in}_{\omega}(I_j)$;  then for all $j=1,\ldots ,r$
there exists $g_j \in I_j$ such that $f=\operatorname{in}_{\omega}
(g_j)$. Thus we have $\prod_j g_j \in \cap_{j} I_j$ and $f^r =
\prod_j
\operatorname{in}_{\omega}(g_j)=\operatorname{in}_{\omega}(\prod_j
g_j) \in \operatorname{in}_{\omega} (\cap_{j} I_j )$. In particular,
$f\in \sqrt{\operatorname{in}_{\omega} (\cap_{j} I_j )}$. This
finishes the proof as the involved ideals are $\omega$--homogeneous.
\end{proof}

The following result is a direct consequence of \cite[Theorem
6.8]{DMM} and Proposition \ref{slopes-primary} when $- \beta \notin
\mathcal{Z}_{\rm primary}(I)$. However, we want to prove it when $-\beta
\notin \mathcal{Z}_{\rm Andean}(I)$ that is a weaker condition.

\begin{theorem}\label{slopes}
Let $I$ be a $A$--graded binomial ideal and consider a binomial
primary decomposition $I=\cap_{\rho , J} C_{\rho ,J}$. If $M_A (I ,
\beta )$ is holonomic (equivalently, $-\beta$ lies outside the
Andean arrangement) then the $L$-characteristic variety of $M_A ( I
, \beta )$ coincide with the union of the $L$-characteristic
varieties of $M_A ( I_{\rho , J}, 0 )$ for all associated toral
primes $I_{\rho ,J}$ of $I$ such that $-\beta \in \qdeg(R/C_{\rho
,J}) $. In particular, the slopes of $M_A (I,\beta )$ along a
coordinate subspace in $\C^n$ coincide with the union of the set of
slopes of $M_A ( I_{\rho , J},0 )$ along the same coordinate
subspace for $I_{\rho ,J}$ varying between all the associated toral
primes of $I$ such that $-\beta \in \qdeg(R/C_{\rho ,J}) $.
\end{theorem}

\begin{proof}
By Proposition \ref{isomorphism-I-beta}, we have that $M_A (I,\beta
)$ is isomorphic to $M_A (I_{\beta} ,\beta )$. We also have that
\begin{equation}
\bigcup_{-\beta \in \qdeg (R/C_{\rho ,J})} \Ch^L (M_A
(C_{\rho ,J} ,\beta )) \subseteq \Ch^L (M_A (I_{\beta},
\beta))\subseteq \mathcal{V}(\operatorname{in}_{L}(I_{\beta}), A
x\xi ) \label{inclusion-charL-Ibeta}
\end{equation}

On the other hand, by Lemma \ref{radical-initial-intersection} we
have that $\mathcal{V}(\operatorname{in}_{L}(I_{\beta}))=\cup
\mathcal{V}(\operatorname{in}_{L}(C_{\rho , J})) =
\mathcal{V}(\operatorname{in}_{L}(I_{\rho , J}))$. Here $\mathcal V$
is the zero set of the corresponding ideal. The result in the
statement follows from the last inclusion, the inclusions
(\ref{inclusion-charL-Ibeta}) and Proposition \ref{slopes-primary}.
\end{proof}

\begin{remark}
Notice that Theorem \ref{slopes} implies that the map from $\C^d
\setminus \mathcal{Z}_{\rm Andean}(I)$ to \emph{Sets} sending
$\beta$ to the set of slopes of $M_A (I,\beta )$ along any fixed
coordinate subspace is upper-semi-continuous in $\beta$.  The
Examples \ref{big-example-1} and \ref{big-example-2} illustrate
Theorem \ref{slopes}. \cite[Theorem 4.1]{ES} has been very useful in
order to construct binomial $A$--graded ideals $I$ starting from
some toral primes that we wanted to be associated primes of $I$.
\end{remark}

\begin{example}\label{big-example-1}
The binomial ideal
$$I= \langle \partial_1\partial_3,\partial_1\partial_4 , \partial_2 \partial_3, \partial_2 \partial_4 ,
\partial_3 \partial_4 ,\partial_1^4 \partial_2^3 - \partial_1 \partial_5 , \partial_1^3 \partial_2^4 -\partial_2 \partial_5,
\partial_3^4-\partial_3 \partial_5 , \partial_4^4- \partial_4 \partial_5^2 \rangle$$ is $A$--graded for the
matrix $$A=\left(\begin{array}{ccccc}
               1 & 0 & 1 & 2 & 3 \\
               0 & 1 & 1 & 2 & 3
             \end{array}\right)$$ Its primary components are the toral
             primes $I_i := I_{\rho_i ,J_i}$, $i=1,2,3,4$, where
             $J_1=\{3,5\}, J_2=\{4,5\}, J_3 =\{ 1,2,5\},J_4=\{5\}$
             and $\rho_i :\ker_{\Z} A_{J_i }\longrightarrow
             \C^{\ast}$ is the trivial character for $i=1,2,3,4$.
             Notice that $\qdeg (R/I_i )=\C A_{J_i}=\C \binom{1}{1}$
             for $i=1,2,4$ and $\qdeg R/I_3 = \C^2$. Using
             Theorem \ref{slopes}, Remark \ref{remark-toral-hypergeometric} and the results in \cite{SW} we have the following:

             If $\beta \in \C^2 \setminus \C \binom{1}{1}$ then $M_A
             (I ,\beta )\simeq M_A (I_3 ,\beta )$ has a unique slope
             $s= 6$ along the hyperplane $\{ x_5 = 0\}$ and it is
             regular along the other coordinate hyperplanes.

If $\beta \in \C \binom{1}{1}$ then $M_A (I ,\beta )$ has the slopes
$s_1= 3/2$, $s_2 =3$ and $s_3 =6$ along $\{x_5 =0\}$.
\end{example}

\begin{example}\label{big-example-2}
The binomial ideal $I=\langle \partial_4 \partial_5 , \partial_3
\partial_5 , \partial_2 \partial_5 , \partial_1 \partial_5, \partial_3 \partial_4 \partial_6 ,\partial_2 \partial_4 \partial_6 ,
\partial_1 \partial_3 \partial_4 ,\partial_1 \partial_2 \partial_4 ,
\partial_5^3 -\partial_6 \partial_5 , \partial_2 \partial_3^2 \partial_4-\partial_3 \partial_4^2 ,
\partial_1^2 \partial_4^3-\partial_4 \partial_6 ,\partial_1^2 \partial_2^2 \partial_3^3 -\partial_3 \partial_6, \partial_1^2 \partial_2^3 \partial_3^2 -\partial_2
\partial_6
\rangle$ is $A$--graded for the matrix
$$A=\left(\begin{array}{cccccc}
                                              1 & 0 & 0 & 0 & 1 & 2 \\
                                              0 & 1 & 0 & 1 & 1 & 2 \\
                                              0 & 0 & 1 & 1 & 1 & 2
                                            \end{array}\right)
$$ Using Macaulay2 we get a primary decomposition of $I$ where the primary components are the toral primes $I_i =I_{\rho_i ,
J_i}$, $i=1,\ldots,6$, where $J_1 =\{  1,2,3,6\}$, $J_2=\{2,3,4\}$,
$J_3 =\{2,4 \}$, $J_4 =\{5,6\}$, $J_5 =\{1,4,6\}$, $J_6=\{1,6\}$ and
$\rho_i :\ker_{\Z} A_{J_i }\longrightarrow \C^{\ast}$ is the trivial
character for $i=1,\ldots,6$.

We have $I_1 =\langle
\partial_1^2 \partial_2^2 \partial_3^2 - \partial_6 , \partial_4
,\partial_5)$, $I_2 =\langle \partial_2 \partial_3 -\partial_4
,\partial_1 ,\partial_5 ,\partial_6 \rangle$, $I_3 = \langle
\partial_1 ,\partial_3,\partial_5 ,\partial_6 \rangle$, $I_4
=\langle \partial_5^2 -\partial_6 ,\partial_1 , \partial_2
,\partial_3 , \partial_4 \rangle$, $I_5 =\langle
\partial_1^2 \partial_4^2 - \partial_6 , \partial_2 ,\partial_3 ,\partial_5
\rangle$ and $I_6 =  \langle \partial_2, \partial_3 ,\partial_4
,\partial_5 \rangle$.

We know that $\qdeg R/I_i =\C A_{J_i}$ for $i=1,\ldots,6$. In
particular, $R/I_1$ is the unique component with Krull dimension
$d=3$.

There are four components with Krull dimension $d-1=2$,
namely $R/I_2 ,R/I_3$ have quasidegrees set $\C A_{J_2} =\C
A_{J_3}=\{y_1 =0 \}\subseteq \C^3$ and $R/I_5 ,R/I_6$ have
quasidegrees set $\C A_{J_5}=\C A_{J_6}=\{ y_2 = y_3  \}\subseteq
\C^3$. There is one component $R/I_4$ with Krull dimension one and
quasidegrees set equal to the line $\C A_{J_4} =
\{y_1=y_2=y_3\}\subseteq \C^3$.

Thus, in order to study the behavior
of $M_A (I,\beta)$ when varying $\beta \in \C^3$ it will be useful
to stratify the space of parameters $\C^3$ by the strata $\Lambda_1
=\C^3 \setminus \{ y_1 (y_2 -y_3 )=0 \}$, $\Lambda_2 =\{y_1 =0
\}\setminus \{y_2 -y_3 =0 \}$, $\Lambda_3 = \{y_2 -y_3 =0
\}\setminus \{y_1 =0 \}$, $\Lambda_4 =\{y_1 = y_2=y_3 \}\setminus
\{0\}$, $\Lambda_5 =\overline{\Lambda_2}\cap
\overline{\Lambda_3}\setminus \{0\}= \{ y_1 =0=y_2 -y_3\}\setminus
\{0\}$ and $\Lambda_6 =\{ 0\}$.

Let us compute the slopes of $M_A (I,\beta )$ along coordinate
hyperplanes according with Theorem \ref{slopes}, Remark
\ref{remark-toral-hypergeometric} and the results in \cite{SW}.
Recall that $a_1\ldots,a_6$ stand for the columns of the matrix $A$.
We have the following situations:

\begin{enumerate}
\item[1)] If $-\beta \in \Lambda_1$ then $R/I_1$ is the unique component whose quasidegrees set contains $-\beta$.
Thus, $M_A (I,\beta )\simeq M_A (I_1 ,\beta)$ has a unique slope
$s=6$ along the hyperplane $\{x_6 =0\}$ because $a_6 /s=[1/3 , 1/3 ,
1/3]^t$ belongs to the plane passing through $a_1 ,a_2 ,a_3$.

\item[2)] If $-\beta \in \Lambda_2$, then $-\beta \in \qdeg (R/I_i )$
if and only if $i\in \{1,2,3\}$ so $M_A (I,\beta )$ has the slope
$s=6$ along $\{ x_6 =0 \}$ arising from $I_1$ and the slope $s=2$
along $\{x_4 = 0\}$ arising from $I_2$ (since $a_4 /2$ lie in the
line passing through $a_2 ,a_3$).

\item[3)] If $-\beta \in \Lambda_3$, then $-\beta \in \qdeg R/I_i$
if and only if $i\in \{1,5,6\}$. $M_A (I,\beta )$ has the slopes
$s=4$ (arising from $I_5$) and $s=6$ (arising from $I_1$) along
$\{x_6 =0\}$.

\item[4)] If $-\beta \in \Lambda_4$, then $-\beta \in \qdeg R/I_i$
if and only if $i\in \{1,4,5,6\}$. $M_A (I,\beta )$ has the slopes
$s=2$ (arising from $I_4$), $s=4$ (arising from $I_5$) and $s=6$
(arising from $I_1$) along $\{x_6 =0\}$.

\item[5)] If $-\beta \in \Lambda_5= \cap_{i \neq 4} \qdeg R/I_i \setminus \qdeg
R/I_4$ then $M_A (I,\beta )$ has the slopes $s=4$ (arising from
$I_5$) and $s=6$ (arising from $I_1$) along $\{x_6 =0\}$ and the
slope $s=2$ (arising from $I_2$) along $\{x_4 =0 \}$.

\item[6)] If $-\beta \in \Lambda_6$ (i.e. $\beta =0$) we have that
$-\beta$ is in the quasidegrees set of all the components $R/I_i$.
Thus, $M_A (I,\beta)$ has the slopes $s=2$ (arising from $I_4$),
$s=4$ (arising from $I_5$) and $s=6$ (arising from $I_1$) along
$\{x_6 =0\}$ and the slope $s=2$ (arising from $I_2$) along $\{x_4
=0 \}$.

\end{enumerate}

In all the cases there are no more slopes along coordinate
hyperplanes. Notice that when we move $-\beta$ from one stratum
$\Lambda_i$ of dimension $r$, $1\leq r \leq d=3$, to another stratum
$\Lambda_j \subseteq \overline{\Lambda_i}$ of dimension $r-1$ then
$M_A (I,\beta)$ can have new slopes along a hyperplane but no slope
disappears.

\end{example}

\begin{remark}
By \cite[Lemma 7.2]{DMM}, all toral primes of a lattice-basis ideal
$I(B)$ have dimension exactly $d$ and are minimal primes of $I (B)$.
Thus, the $L$-characteristic varieties and the set of slopes of $M_A
(I(B),\beta)$ are independent of $-\beta \notin {\mathcal Z}_{{\rm
Andean}}(I(B))$.
\end{remark}

To finish this Section we are going to compute the multiplicities of
the $L$--characteristic cycle of a holonomic binomial $D$--module
$M_A(I,\beta)$ for $\beta$ generic. Recall that the volume
$\operatorname{vol}_{\Lambda} (B)$ of a matrix $B$ with columns $b_1
,\ldots ,b_k \in \Z^d$ with respect to a lattice $\Lambda\supseteq
\Z B$ is nothing but the Euclidean volume of the convex hull of $\{
0 \}\cup \{b_1 ,\ldots , b_k\}$ normalized so that the unit simplex
in the lattice $\Lambda$ has volume one.

From now on we assume that $\mathcal{Z}_{\rm Andean}(I)\neq \C^d$
and that $\beta \in \C^d$ is generic. In particular we assume that
all the quotients $R/C_{\rho ,J}$ whose quasidegrees set contain
$-\beta$ are toral and have Krull dimension $d$. The generic
condition will also guarantee that $\beta$ is not a rank--jumping
parameter of any hypergeometric system $\mathcal{H}_0 (E-\beta ,
I_{\rho ,J})$.

Under this assumptions it is proved in \cite[Theorem 6.10]{DMM} that
the holonomic rank of $M_A (I,\beta )$ equals
$$\operatorname{rank}(M_A (I,\beta ))=\sum_{R/ I_{\rho ,J} \mbox{
toral $d$--dimensional }} \mu_{\rho ,J} \operatorname{vol}_{\Z A_J}
(A_J )$$

We will use the same strategy in order to compute the multiplicities
in the $L$-characteristic cycle $\CCh^L(M_A (I,\beta))$. It is
enough to compute the multiplicities in the $L$-characteristic cycle
of $M_A (C_{\rho ,J} ,\beta)$ for each $d$--dimensional toral
component $C_{\rho ,J}$ of $I$ and then apply \cite[Theorem
6.8]{DMM}.

In \cite[Section 3.3]{SW} the authors give an index formula for the
multiplicity $\mu_{A,0}^{L,\tau}(\beta)$ of the component
$\overline{C_A^\tau}$ in the $L$--characteristic cycle
$\CCh^L(M_A(\beta))$ of a hypergeometric $D$--module; see equality
(\ref{L-characteristic-variety}). They prove that these
multiplicities are independent of $\beta$ if $\beta$ is generic (see
\cite[Theorem 4.28]{SW}). Let us denote by $\mu_A^{L,\tau}$ this
constant value.

If $M$ is a finitely generated $R$--module, we denote by
$\mu_{A,0}^{L,\tau}(M,\beta)$ the multiplicity of the component
$\overline{C_A^\tau}$ in the $L$--characteristic cycle
$\CCh^L({\cH}_0(E-\beta,M))$ (see \cite[Definition 4.7]{SW}).

For $J\subset \{1,\ldots,n\}$ we denote $A_J$ the submatrix whose
columns are indexed by $J$, $D_J$ the Weyl algebra with variables
$\{x_j,\partial_j\, \vert \, j\in J$  and $L_J$ the filtration on
$D_J$ induced by the weights $(u_j,v_j)$ for $j\in J$. In
particular, we can define the multiplicity $\mu_{A_J}^{L_J,\tau}$
for any face $\tau$ of the $(A_J,L_J)$--umbrella $\Phi_{A_J}^{L_J}$.

\begin{theorem}\label{mult-L-char-cycles}
Let $R/C_{\rho ,J}$ be a toral $d$--dimensional module and let
$\beta$ be generic. We have for all $\tau \in \Phi_{A_J}^{L_J}$ and
for any filtration $L$ on $D$ that
$$\mu_{A,0}^{L,\tau}(R/C_{\rho ,J} ,\beta )= \mu_{\rho ,J}
\mu_{A_J }^{L_J , \tau }.$$
\end{theorem}

\begin{proof} It follows the ideas of the last part of the proof of \cite[Theorem 6.10]{DMM}
(see also the proof of Theorem \ref{regularity}). We write
$M=R/C_{\rho ,J}$ and consider a toral filtration $M_0=(0) \subseteq
M_1 \subseteq \cdots \subseteq M_r=M$ each successive quotient
$M_i/M_{i-1}$ being isomorphic to $\frac{R}{I_{\rho_i,
J_i}}(\gamma_i)$ for some $\gamma_i\in \ZZ^d$. The number of
successive quotients of dimension $d$ is the multiplicity
$\mu_{\rho,J}$ of the ideal $I_{\rho,J}$ in $C_{\rho,J}$. From the
assumption on $\beta$ we can take $-\beta$ outside  the union of the
quasidegree sets of $\frac{R}{I_{\rho_i, J_i}}$ with Krull dimension
$< d$. Then
$$\cH_j(E-\beta,M_i/M_{i-1}) = \left\{\begin{array}{lcl} 0 &
{\mbox{ if }} & I_{\rho_i,J_i} \neq
I_{\rho,J}\\
\cH_j(E-\beta+\gamma_i,R/I_{\rho,J})(\gamma_i) & {\mbox{ otherwise.
}} & \end{array} \right.$$ Again using that $\beta$ is generic, we
have that $\cH_j(E-\beta , M_i /M_{i-1})=0$ for any $i$ and any
$j\geq 1$. The statement of the Theorem follows by applying
decreasing induction on $i$ and the additivity of
$\mu_{A,0}^{L,\tau}$ with respect to the exact sequence
$$0\longrightarrow \cH_0 (E-\beta, M_{i-1} ) \longrightarrow \cH_0 (E-\beta, M_i ) \longrightarrow \cH_0 (E-\beta, M_i /
M_{i-1}) \longrightarrow 0. $$ We notice here that the multiplicity
$\mu_A^{L,\tau}$ for $\cH_0(E-\beta+\gamma_i, R/I_{\rho,J})$ equals
$\mu_{A_J}^{L_J,\tau}$ for the hypergeometric $D_J$--module
$M_{A_J}(\beta-\gamma_i)$ because $\beta$ is generic.
\end{proof}

\section{On the Gevrey solutions and the irregularity of binomial $D$--modules}\label{Section-Gevrey}

Let us denote by $Y_i$ the hyperplane $x_i=0$ in $\CC^n$. Again by
\cite[Theorem 6.8]{DMM}, in order to study the Gevrey solutions and
the irregularity of a holonomic binomial $D$--module $M_A (I,\beta)$
for generic parameters $\beta \in \C^d$ it is enough to study each
binomial $D$--module $M_A (C_{\rho ,J} ,\beta)$ arising from a
$d$--dimensional toral primary component $R/C_{\rho ,J}$. For any
real number $s$ with $s\geq 1$, we consider, the irregularity
complex of order $s$, $\operatorname{Irr}_{Y_i}^{(s)} ({M}_A
(C_{\rho ,J} ,\beta))$ (see \cite[Definition 6.3.1]{Meb90}). Since
$M_A(C_{\rho,J},\beta)$ is holonomic, by a result of Z. Mebkhout
\cite[Theorem 6.3.3]{Meb90} this complex is a perverse sheaf and
then for $p\in Y_i$ generic it is concentrated in degree 0.

For $r\in \RR$ with $r\geq 1$
we denote by $L_r$ the filtration on $D$ induced by $L_r=F+(r-1)V$ and we will write simply $\Phi_A^r$ instead of $\Phi_A^{L_r}$ and $\mu_{A,0}^{r,\tau}$ instead of  $\mu_{A,0}^{L_r,\tau}$.

\begin{theorem}\label{theorem-dim-irreg}
Let $R/C_{\rho ,J}$ be a toral $d$--dimensional module, $\beta$
generic, $p\in Y_i$ generic, $i=1,\ldots ,n$ and $s$ a real
number with $s\geq 1$. We have that
$$\dim_{\C} \mathcal{H}^0 \left(\operatorname{Irr}_{Y_i}^s ({M}_A (C_{\rho ,J} ,\beta))\right)_p = \mu_{\rho , J} \sum_{i \notin \tau \in \Phi_{A_J}^{s}\setminus
\Phi_{A_J}^{1}} \operatorname{vol}_{\Z A_J } (A_{\tau})$$
\end{theorem}

\begin{proof} We follow the argument of the proof of Theorem 7.5 in \cite{F}.
We apply results of Y. Laurent and Z. Mebkhout
\cite[Lemme 1.1.2 and Section 2.3]{Laurent-Mebkhout} to get
$$\dim_{\C} \mathcal{H}^0 \left(\operatorname{Irr}_{Y_i}^s ({M}_A (C_{\rho ,J} ,\beta))\right)_p = \mu_{A,0}^{s+\epsilon,\emptyset} - \mu_{A,0}^{1+\epsilon,\emptyset} + \mu_{A,0}^{1+\epsilon,\{i\}} - \mu_{A,0}^{s+\epsilon,\{i\}}.$$
To finish the proof we apply Theorem \ref{mult-L-char-cycles}
and Theorem 7.5 \cite{F}.
\end{proof}

\begin{remark}
Notice that the above formula for $\dim_{\C} \mathcal{H}^0
(\operatorname{Irr}_{Y_i}^s ({M}_A (C_{\rho ,J} ,\beta
)))_p=0$ yields zero if $i\notin J$ since in that case the induced
filtration $(L_s)_J$ (denoted just by $s$ by abuse of notation) is
constant and so $\Phi_{A_J}^{s}\setminus \Phi_{A_J}^{1}=\emptyset$.
\end{remark}

Let us see how to compute Gevrey solutions of a binomial $D$--module
$M_A(I,\beta)$. By (3.3) in \cite{DMM2} the $I_{\rho , J}$--primary
component $C_{\rho ,J}$ of an irredundant primary decomposition of
any $A$--graded binomial ideal $I$ (for some minimal associated
prime $I_{\rho ,J}=I_{\rho}+ \mathfrak{m}_J$ of $I$) contains
$I_{\rho}$. Thus,
\begin{equation}
I_{\rho} + \mathfrak{m}_J^r \subseteq C_{\rho ,J} \subseteq
\sqrt{C_{\rho ,J}} = I_{\rho ,J} = I_{\rho} + \mathfrak{m}_J
\label{contention1}\end{equation} for sufficiently large integer
$r$. In fact, it is not hard to check that $C_{\rho,J}= I_{\rho} +
B_{\rho ,J}$ for some binomial ideal $B_{\rho ,J}\subseteq R$ such that $\mathfrak{m}_J^r \subseteq B_{\rho ,J}
\subseteq \mathfrak{m}_J$. Let us fix such an ideal $B_{\rho,J}$.

For any monomial ideal $\mathfrak{n}\subseteq C_{\rho ,J}$ such that
$\sqrt{\mathfrak{n}}=\mathfrak{m}_J$ we have that $$H_A (I_{\rho}
+\mathfrak{n}, \beta )\subseteq H_A (C_{\rho ,J} ,\beta ) \subseteq
H_A (I_{\rho ,J},\beta ).$$ Let us fix such an ideal $\mathfrak n$.
In particular, any formal solution of $M_A (I_{\rho ,J},\beta )$ is
a solution of $M_A (C_{\rho , J},\beta )$ and any solution of $M_A
(C_{\rho , J},\beta )$ is a solution of $M_A (I_{\rho} +\mathfrak{n}
,\beta )$.

Let us assume that $C_{\rho ,J}$ is toral (i.e. $ R/I_{\rho,J}$ has
Krull dimension equal to $\rank A_J$). We will also assume that
$\rank A_J = \rank A$ in order to ensure that $\qdeg ( R/C_{\rho ,J}
)=\C^d$.

On the one hand, both the solutions of $M_A (I_{\rho, J} ,\beta )$
and the solutions of $M_A (I_{\rho} +\mathfrak{n} ,\beta )$ can be
described explicitly if the parameter vector $\beta$ is generic
enough. More precisely, a formal solution of the hypergeometric
system $M_A (I_{\rho , J} ,\beta )$ with very generic $\beta$ is
known to be of the form
$$\phi_v =\sum_{u \in \ker A_J \cap \Z^J } \rho (u)
\dfrac{(v)_{u_{-}}}{(v+u)_{u_{+}}} x_J^{v+u}$$ where $v\in \C^J$
such that $A_J v=\beta$ and $(v)_w =\prod_{j \in J} \prod_{0 \leq
i\leq w_j -1}(v_j -i)$ is the Pochhammer symbol (see \cite{GZK89,
SST}).  Here, $v$ needs to verify additional
conditions in order to ensure that $\phi_v$ is a formal series along
a coordinate subspace or a holomorphic solution.

The vectors $v$ you need to consider to describe a basis of the
space of Gevrey solutions of a given order along a coordinate
subspace of $\C^n$ for the binomial $D$-module $M_A (I_{\rho , J}
,\beta )$ are the same that are described in \cite{F} for the
hypergeometric system $M_{A_J} (\beta )$.

On the other hand, for $\gamma$ in $\N^{\overline{J}}$ let
$G_{\gamma}$ be either a basis of the space of holomorphic solutions
near a non singular point or the space of Gevrey solutions of a
given order along a coordinate hyperplane of $\C^J$ for the system
$M_{A_J} (I_{\rho},\beta - A_{\overline{J}}\gamma)$, where
$\overline{J}$ denotes the complement of $J$ in $\{1,\ldots ,n\}$
and $x_{\overline{J}}^{\gamma}$ runs in the set
$\operatorname{S}_{\overline{J}} (\mathfrak{n})$ of monomials in $\C
[x_{\overline{J}}]$ annihilated by the monomial differential
operators in $\mathfrak{n}$. Then a basis of the same class of
solutions for the system $M_A (I_{\rho} +\mathfrak{n} ,\beta )$ is
given by $$\mathcal{B}=\{ x_{\overline{J}}^{\gamma} \varphi : \;
x^{\gamma} \in \operatorname{S}_{\overline{J}} (\mathfrak{n}), \;
\varphi \in G_{\gamma}\}$$

We conclude that any holomorphic or formal solution of $M_A (C_{\rho
,J},\beta )$ can be written as a linear combination of the series in
$\mathcal{B}$. The coefficients in a linear combination of elements
in $\mathcal{B}$ that provide a solution of $M_A (C_{\rho ,J},\beta
)$ can be computed if we force a general linear combination to be
annihilated by the binomial operators in a set of generators of
$B_{\rho ,J}$ that are not in $\mathfrak{n}$.

Thus, the main problem in order to compute formal or analytic
solutions of $M_A (C_{\rho ,J},\beta )$ is that the ideal $B_{\rho,
J}$ is not a monomial ideal in general and that a minimal set of
generators may involve some variables $x_j$ for $j\in J$. Let us
illustrate this situation with the following example.

\begin{example} Let us write $x=x_1, y=x_2, z=x_3, t=x_4$ and consider the
binomial ideal $C_{\rho,J}= I_{\rho} + B_{\rho ,J}\subseteq \C
[\partial_x,\partial_y ,\partial_ z,\partial_t]$ where $J=\{1,2\}$,
$\rho : \ker (A_J )\cap \Z^2 \rightarrow \C^{\ast}$ is the trivial
character, $A$ is the row matrix $(2,3,2,2)$, $I_{\rho}= \langle
\partial_x^3 -\partial_y^2\rangle$ and $B_{\rho ,J}=\langle \partial_z^2
-\partial_x \partial_t , \partial_t^2 \rangle$.

Notice that $C_{\rho,J}$ is $A$-graded for the row matrix $A= (2 \;
3 \; 2\; 2)$ and that $C_{\rho,J}$ is toral and primary.

Since $C_{\rho,J}$ is primary and its radical ideal is $I_{\rho} +
\mathfrak{m}_J =\langle
\partial_x^3 -\partial_y^2 ,\partial_z ,\partial_t \rangle$, we have that $M_A (C_{\rho ,J},\beta )$ is an
irregular binomial $D$-module for all parameters $\beta \in \C$ (see
Theorem \ref{regularity}) and that it has only one slope $s=3/2$
along its singular locus $\{y=0\}$.

We are going to compute the Gevrey solutions of $M_A (C_{\rho
,J},\beta )$ corresponding to this slope.

By the previous  argument and using that $\mathfrak{n}=\langle
\partial_z^4 , \partial_t^2 \rangle \subseteq B_{\rho ,J}$ we obtain that any
Gevrey solution of $M_A (C_{\rho ,J},\beta )$ along $\{y=0\}$ can be
written as
$$f=\sum_{\gamma,k} \lambda_{\gamma,k} z^{\gamma_z } t^{\gamma_t} \phi_k (\beta - 2 \gamma_z
-2 \gamma_t )$$ where $\lambda_{\gamma,k}\in \C$, $\gamma =(\gamma_z
,\gamma_t )$, $\gamma_z \in \{0,1,2,3\}$, $\gamma_t , k \in \{0,1\}$
and $$\phi_k (\beta - 2 \gamma_z -2 \gamma_t )=\sum_{m\geq
0}\dfrac{((\beta -3 k )/2 - \gamma_z -  \gamma_t)_{3m}}{(k+2 m)_{2
m}} x^{(\beta -3 k )/2 - \gamma_z -  \gamma_t - 3 m} y^{k +2 m}$$ is
a Gevrey series of index $s=3/2$ along $y=0$ at any point $p\in
\{y=0\}\cap \{x\neq 0\}$ if $(\beta - 3 k)/2 -\gamma_z -  \gamma_t
\notin \N$.

We just need to force the condition $\partial_x \partial_t
(f)=\partial_z^2 (f)$ in order to obtain the values of
$\lambda_{\gamma ,k}$ such that $f$ is a solution of $M_A
(C_{\rho,J},\beta )$.

In this example, we obtain the conditions
$\lambda_{(2,1),k}=\lambda_{(3,1),k}=0$ for $k=0,1$ and
$$\lambda_{(\gamma_z +2 , 0), 1} = \dfrac{((\beta -3 k)/2 - \gamma_z
)}{(a+1)(a+2)} \lambda_{(\gamma_z , 1), k}
$$ for $k , \gamma_z = 0,1$.

In particular we get an explicit basis of the space of Gevrey
solutions of $M_A (C_{\rho ,J},\beta )$ along $y=0$ with index equal
to the slope $s=3/2$ and we have that the dimension of this space is
$8$. Notice that $8=4 \cdot 2$ is the expected dimension (see
Theorem \ref{theorem-dim-irreg}) since $\mu_{\rho ,J}= 4$ and the
dimension of the corresponding space for $M_A (I_{\rho ,J},\beta )$
is $2$ (see \cite{FC1, FC2}).
\end{example}

\end{document}